\documentclass[ejs,preprint,dvips,twoside,linksfromyear]{imsart} %

\RequirePackage[colorlinks,citecolor=blue,urlcolor=blue]{hyperref}
\usepackage[latin1]{inputenc}
\usepackage{dsfont}
\usepackage{mathrsfs}
\usepackage{amsmath}
\usepackage{amssymb}
\usepackage{verbatim}
\usepackage{amsthm}
\usepackage{graphicx}
\usepackage[usenames]{color}
\usepackage[dvipdf,dvips,rgb,dvipsnames,svgnames,table,x11names]{xcolor}
\usepackage{fancybox}
\usepackage[square,numbers]{natbib}
\usepackage{enumitem}
\def\a{\alpha}
\def\b{\beta}
\def\g{\gamma}

\def\d{\delta}

\def\e{\epsilon}

\def\l{\lambda}

\def\s{\sigma}
\def\S{\Sigma}

\def\ie{\textit{i.e., }}

\def\cf{\textit{cf. }}
\def\NN{\mathbb N}

\def\RR{\mathbb R}

\def\ldeux{\mathbb{L}_2}
\def\lp{\mathbb{L}_p}
\def\card{\textbf{Card}}
\def\fcar{\mathds{1}}
\def\pinf{{+\infty}}
\def\suchthat{\,|\,}

\renewcommand\Re{\operatorname{Re}}
\renewcommand\Im{\operatorname{Im}}
\def\ii{\textrm{i}}
\def\esp{\mathbf E}
\def\var{\mathbf Var}

\def\prob{\mathbf P}
\def\calN{\mathcal N}
\def\simiid{\overset{iid}{\sim}}

\def\nzeroun{\mathcal{N}(0,1)}
\theoremstyle{plain}
\newtheorem{theorem}{Theorem}
\newtheorem{lemma}{Lemma}
\newtheorem{proposition}{Proposition}

\newtheorem*{theorem*}{Theorem}
\newtheorem*{lemma*}{Lemma}
\newtheorem*{proposition*}{Proposition}
\newtheorem*{corollary*}{Corollary}
\theoremstyle{remark}
\newtheorem{remark}{Remark}
\newtheorem*{remark*}{Remark}

\newtheorem*{note*}{Note}
\theoremstyle{definition}

\newtheorem*{definition*}{Definition}
\def\Ydiese{Y^{\texttt{\#}}}
\def\cdiese{c^\texttt{\#}}
\def\xidiese{\xi^{\texttt{\#}}}
\def\FsL{\mathcal{F}_{s,L}}
\def\bu{\boldsymbol{u}}
\def\bc{\boldsymbol{c}}
\def\bu{\boldsymbol{u}}
\def\bcdiese{\boldsymbol{c}^\texttt{\#}}

\def\hzero{H_0}
\def\hun{H_1}
\def\rhosigma{\rho_\sigma}
\def\thetazero{\Theta_0}
\def\thetaun{\Theta_1}
\def\pccdiese{\prob_{\boldsymbol{c},\boldsymbol{c}^\texttt{\#}}}
\def\Eccdiese{\esp_{\boldsymbol{c},\boldsymbol{c}^\texttt{\#}}}
\def\bY{\boldsymbol{Y}}
\def\bYdiese{\boldsymbol{Y}^{\texttt{\#}}}

\def\bepsilon{\boldsymbol{\epsilon}}

\def\somNsigma{\sum_{j=1}^{N_\s}}

\def\fdiese{f^\texttt{\#}}

\def\bxi{\boldsymbol{\xi}}
\def\btildexi{\tilde{\boldsymbol{\xi}}}
\def\sdiese{\sigma^\texttt{\#}}

\begin{document}

\begin{frontmatter}
\title{Minimax hypothesis testing for curve registration}
\runtitle{Minimax hypothesis testing for curve registration}

\begin{aug}
\author{\fnms{Olivier} \snm{Collier}\ead[label=e1]{olivier.collier@imagine.enpc.fr}}
\runauthor{O. Collier}
\address{IMAGINE, LIGM, Université Paris Est, Ecole des Ponts ParisTech, France and ENSAE, CREST}
\end{aug}

\begin{abstract}
This paper is concerned with the problem of goodness-of-fit for curve
registration, and more precisely for the shifted curve model, whose
application field reaches from computer vision and road traffic
prediction to medicine. We give bounds for the asymptotic minimax
separation rate, when the functions in the alternative lie in Sobolev
balls and the separation from the null hypothesis is measured by the
$l_2$-norm. We use the generalized likelihood ratio to build a
nonadaptive procedure depending on a tuning parameter, which we choose
in an optimal way according to the smoothness of the ambient space.
Then, a Bonferroni procedure is applied to give an adaptive test over a
range of Sobolev balls. Both achieve the asymptotic minimax separation
rates, up to possible logarithmic factors.
\end{abstract}

\begin{keyword}
\kwd{Adaptive testing}
\kwd{composite null hypothesis}
\kwd{generalized maximum likelihood}
\kwd{minimax hypothesis testing}
\end{keyword}

\received{\smonth{1} \syear{2012}}

\end{frontmatter}

\section*{Introduction}\label{section:introduction}

\subsection*{Curve registration}

Our concern is the statistical problem of curve registration, which
appears naturally in a large number of applications, when the available
data consist of a set of noisy, distorted signals that possess a common
structure or pattern. This pattern constitutes the essential
information that we want to dig out from the observations. However, the
deformations of the signals are generally nonlinear and relatively
complex, which complicates the statistical task. Fortunately it is
relevant in some cases to assume that the signals only differ from each
other by a horizontal shift: we call this modeling the shifted curve
model. For instance, it was successfully adopted for the interpretation
of the ElectroCardioGramms: each deflection is considered as a
repetition of the same signal starting at a random time.
\citet{TriganoIsserlesRitov2011} proposed an estimator of the common
pattern. Interestingly, the assumptions on the deformations are in
practice violated due to the baseline wandering, a periodic vertical
perturbation of the potential, but the estimation of the structural
pattern performs well yet.

By contrast, SIFT descriptors (\cf\citet*{Lowe2004}) in computer vision
are an example where the specification of the deformations is
essential: selected keypoints of an image are assigned with descriptors
including a histogram of the local gradient. If the image is rotated,
the histogram of each keypoint is simply shifted by the angle of the
rotation. To match the keypoints of the two images, it is then
sufficient to test the adequation of their histograms with the shifted
curve model. So, testing the model is sometimes the main concern, and
even when estimation matters, the adequation of the model may have to
be tested, as the estimation techniques depend on the structure of the
deformations.

We refer to the papers \citet*{BigotGadat2010},
\citet*{BigotGadatLoubes2009}, \citet*{BigotGamboaVimond2009},
\citet*{CastilloLoubes2009}, \citet*{DalalyanGolubevTsybakov2006},
\citet*{Dalalyan2007} and \citet*{GamboaLoubesMaza2007} for results on
the estimation of different features of the curve registration model.
The present work builds on \citet*{CollierDalalyan2012}, where a
comprehensive overview can be found.

\subsection*{Shifted curve model}

This paper deals with the shifted curve model, which we will state in a
Gaussian sequence form, but which originally relates on two
$2\pi$-periodic functions $f$ and $\fdiese$ in $\ldeux$. Expanding
these functions in the complex Fourier basis, we get
\begin{equation*}
f(t) = \sum_{j=-\infty}^\pinf c_j(f) e^{\ii j t} \text{ and } \fdiese(t) = \sum_{j=-\infty}^\pinf c_j(\fdiese) e^{\ii j t} \text{ for $t\in[0,2\pi]$,}
\end{equation*}
where $c_j(f) = \frac{1}{2\pi} \int_0^{2\pi} f(t) e^{-\ii j t} \,dt$
and $c_j(\fdiese) = \frac{1}{2\pi} \int_0^{2\pi} \fdiese(t) e^{-\ii j
t} \,dt$.

With this notation, if $f$ and $\fdiese$ only differ from each other by
a shift, then the Fourier coefficients verify $c_j(\fdiese) = e^{\ii j
\tau} c_j(f)$, for some real $\tau$ in $[0,2\pi]$ and all non-zero
integers $j$. Hence, if we introduce the pseudo-distance $d$ such that
\begin{equation}\label{distance}
d^2((c_1,\ldots),(\cdiese_1,\ldots)) \triangleq \inf_\tau \sum_{j=1}^\pinf |c_j - e^{-\ii j\tau}\cdiese_j|^2,
\end{equation}
and given that $c_j(f) = \overline{c_{-j}(f)}$ for every integer $j$,
testing that $f$ was shifted from $\fdiese$ amounts to testing if
$d(\bc,\bcdiese)=0$ where $\bc=(c_1(f),c_2(f),\ldots)$ and
$\bcdiese=(c_1(\fdiese),c_2(\fdiese),\ldots)$.

Now, if we assume that the observations are given by the white noise model
\begin{equation*}
dY(t) = f(t)\,dt + \s\,dW(t) \text{ and } d\Ydiese(t) = \fdiese(t)\,dt + \s\,dW^\texttt{\#}(t),
\end{equation*}
where $\s>0$ and $W,W^\texttt{\#}$ are independent Wiener processes, we
can state our model in a more convenient Gaussian sequence form:
\begin{flushleft}
\begin{equation}\label{model}
\begin{cases}
\ Y_j = c_j + \s\xi_j \\
\ \Ydiese_j = \cdiese_j + \s\xidiese_j
\end{cases}\!\!\!\!,\qquad j=1,2,\ldots ,
\end{equation}
\end{flushleft}
where
\begin{itemize}[align=left, leftmargin=*, noitemsep]
\item $\{\xi_j,\xidiese_j;j=1,2,\ldots\}$ is a family of independent complex random variables, whose real and imaginary parts are independent standard Gaussian variables,
\item $\s$ is assumed to be known.
\end{itemize}
\eject

\noindent
Our problem amounts to testing $\hzero$ against $\hun$ with
\begin{equation}\label{hypotheses}
\begin{cases}
\ \hzero: \; d(\bc,\bcdiese) = 0, \\
\ \hun: \; d(\bc,\bcdiese) \geq C \rhosigma,
\end{cases}
\end{equation}
where $C$ is a positive constant and $\rhosigma$ is a sequence of
positive real numbers. For reasons that we shall explain later, we
assume that $\bc$ and $\bcdiese$ belong under the alternative to a
Sobolev ball
\begin{equation}\label{FsL}
\FsL\triangleq\Big\{\bu=(u_1,u_2,\ldots):\|\bu^{(s)}\|_2^2\triangleq\sum_{j=1}^{\infty} j^{2s} |u_j|^2 \leq L^2\Big\},
\end{equation}
with $s>0$. With this notation, we denote $\thetazero$ and $\thetaun$
the parameter sets corresponding to the hypotheses $\hzero$ and $\hun$,
$\bY$ and $\bYdiese$ the sequences $(Y_1,Y_2,\ldots)$ and
$(\Ydiese_1,\Ydiese_2,\ldots)$, and we call $\pccdiese$ the probability
engendered by $(\bY,\bYdiese)$ when the parameters are $\bc$ and
$\bcdiese$.

A detailed discussion of the model is deferred to
Section~\ref{section:discussion}, but before that, we point out that
our choice of the Gaussian sequence model is not restrictive, since
this model is equivalent in Le Cam's sense to many other models,
including Gaussian white noise, density estimation (\cf
\citet*{Nussbaum1996}), nonparametric regression (\cf
\citet*{BrownLow1996}, in the case of random design in
\citet*{Reiss2008}, in the case of nonGaussian noise in
\citet*{GramaNussbaum1998} and \citet*{GramaNussbaum2002}), ergodic
diffusion (\cf \citet*{DalalyanReiss2006}). On the other hand, the
Gaussian noise is accepted in computer vision as a good approximation
of the Poisson noise, that is more natural in this context.

\subsection*{Minimax testing}

A randomized test in our model is a random variable taking values in
$[0,1]$ and measurable with respect to the $\s$-algebra engendered by
$(\bY,\bYdiese)$. In practice, the user simulates an independent random
variable with a Bernoulli distribution of parameter the value of the
test, which was computed from the data $(\bY,\bYdiese)$. The null
hypothesis is accepted, respectively rejected, when the result of the
simulation is $0$ or $1$. We say that a test is nonrandomized when it
only takes the values $0$ or $1$.

To measure the performance of a test $\psi$, we choose the minimax
point of view, in which the errors of first and second kind are defined
by
\begin{equation}
\begin{cases}
\ \a(\psi,\thetazero) = \sup_{\thetazero} \Eccdiese\big(\psi\big), \\
\ \b(\psi,\thetaun) = \sup_{\thetaun} \Eccdiese\big(1-\psi\big).
\end{cases}
\end{equation}
Note that in the nonrandomized case, $\a(\psi,\thetazero) = \sup_{\thetazero} \pccdiese(\psi=1)$
and $\b(\psi,\thetaun) = {\sup_{\thetaun} \pccdiese(\psi=0)}$.

We say that consistent testing in the asymptotic minimax sense is
possible if for all $\a,\b > 0,$ there exists a test $\psi_\s$ such
that
\begin{equation}
\begin{cases}
\ \varlimsup\limits_{\s\to0} \a(\psi_\s,\thetazero) \leq \a, \\
\ \varlimsup\limits_{\s\to0} \b(\psi_\s,\thetaun) \leq \b.
\end{cases}
\end{equation}
The distance between the null and the alternative hypotheses,
$C\rhosigma$, determines the existence of such tests. Indeed, if
$C\rhosigma$ is too small, no testing procedure is asymptotically
better than a blind guess, for which
$\a(\psi,\thetazero)+\b(\psi,\thetaun) = 1$. For a fixed pair $\a,\b$,
we call $\rhosigma^*$ the asymptotic minimax separation rate if there
are two positive constants $C_*$ and $C^*$ such that consistent testing
is impossible for $\rhosigma=\rhosigma^*$ and $C<C_*$, and possible for
$\rhosigma=\rhosigma^*$ and $C>C^*$. The best constants $C_*$ and $C^*$
satisfying these conditions are called exact separation constants.
Conventionally, one applies the informal minimal writing length rule to
avoid nonuniqueness of the minimax separation rate and of these
constants. Moreover, a test which is consistent when
$\rhosigma=\rhosigma^*$ and for some $C>0$ is called asymptotically
minimax rate optimal.

There is a vast literature on the subject of minimax testing: minimax
separation rates were investigated in many models, including the
Gaussian white noise model, the regression model, the Gaussian sequence
model and the probability density model, for the greater part in signal
detection, \ie testing the hypothesis ``$f\equiv0$'' against the
alternative ``$\|f\|\geq C\rhosigma$''. We present a selective overview
of the papers that are the most relevant in the context of this work.

Starting from \citet*{Ingster1982}, \citet*{Ermakov1990} and
\citet*{Ermakov1996}, where the minimax separation rate and the exact
separation constants were obtained when the functions in the
alternative lie in ellipsoids and the separation from $0$ is measured
by the $l_2$-norm, various cases were considered: $l_p$-bodies as well
as Sobolev, Hölder and Besov classes. We refer to
\citet*{IngsterSuslina2003} and \citet*{Ingster1993} for a survey. The
cases when the functions in the alternative set lie in Sobolev or
Hölder classes and the separation from $0$ is measured by the sup-norm
or by their values at a fixed point were studied in
\citet*{LepskiTsybakov2000}. Finally, the case of the $\lp$-norm with
$p<2$ in Besov classes was considered in \citet*{LepskiSpokoiny1999}.

Now, all the previously cited results are asymptotic, in the sense that
the noise level $\s$ (in the white noise model) tends to $0$. But from
a practical point of view, it may be interesting to look at the problem
from a nonasymptotic point of view. In the regression and Gaussian
sequence models, \citet*{Baraud2002} derived nonasymptotic minimax
separation rates when the functions in the alternative lie in
$l_p$-bodies (${0<p\leq2}$) and the separation from $0$ is measured by
the $l_2$-norm. \citet*{BaraudHuetLaurent2003,BaraudHuetLaurent2005}
proposed procedures for testing linear or convex hypotheses in the
regression model, and \citet*{FromontLevyLeduc2006} inspected the
improvement implied by a further hypothesis on the periodicity of the
signal in the periodic Sobolev balls.

\subsection*{Composite null hypothesis testing}

Up to here, we have reviewed results dealing mainly with a simple null
hypothesis, namely in the case of signal detection: ``$f\equiv0$''. In
contrast, the testing problem in the shifted curve model deals with a
composite null hypothesis. Here, we give a brief overview of the papers
presenting hypothesis testing problems with composite null hypotheses.

The series of papers \citet*{Baraud2002},
\citet*{BaraudHuetLaurent2003,BaraudHuetLaurent2005} tackled the case
of a nonparametric null hypothesis, but their assumptions are not
applicable in our set-up, since our null hypothesis as defined in
(\ref{hypotheses}) is neither linear nor convex. On the other hand, the
test of a parametric model against a nonparametric one was studied in a
substantial number of papers (\cf\citet*{HorowitzSpokoiny2001} and
references therein), but only in \citet*{HorowitzSpokoiny2001} from a
minimax point of view. The minimax separation rate that they obtained
is the same as with a simple null hypothesis. This is due to the strong
assumptions made on the behaviour of the estimator of the parameter
characterizing the model under $\hzero$.

On a related note, \citet*{GayraudPouet2001,GayraudPouet2005} treated a
more general composite null hypothesis in the regression model, that is
mainly characterized by its entropy. In fact, the set of functions in
the null hypothesis can grow with the sample size, and so be
nonparametric. Their rate is the same as in the case of a simple
hypothesis. Finally, \citet*{ButuceaTribouley2006} also considers the
case of a nonparametric null hypothesis, since $\hzero$ is ``$f=g$'',
where $f$ and $g$ are two density functions.

\subsection*{Adaptive testing}

A limitation of the minimax approach is that the optimal tests depend
on the smoothness class. This is not convenient from a practical point
of view, because the choice of the smoothness seems to be unnatural and
arbitrary. To obtain handier procedures, we need an adaptive definition
for hypothesis testing.

Prior to testing, some sets of smoothness parameters $s,L$ must be
chosen, over which adaptation is performed. Typically, these sets are
taken as compact intervals $[s_1,s_2]$, $[L_1,L_2]$. To each couple of
smoothness parameters $(s,L)$, we associate the smoothness set
$\mathcal{F}_{s,L}$, and we write $\thetazero^{s,L}$ and
$\thetaun^{s,L}$ the corresponding null and alternative hypotheses.
Note that, in our problem, $\thetazero^{s,L}\equiv\thetazero$ is
independent of the smoothness parameters, and that $\thetaun^{s,L}$
depends on $(s,L)$, not only because $\bc$ and $\bcdiese$ are in
$\FsL$, but also since $\rhosigma$ is allowed to be a function of $s$:
as a matter of fact, $\thetaun^{s,L}$ depends on the choice of the
radius $C\rhosigma(s)$. The easiest way to achieve adaptation is to use
the test corresponding to the most constraining smoothness $(s_1,L_2)$,
but this entails a significant loss of efficiency if the tested
parameters are in fact smoother.

Thus, we prefer a more economical approach and we will say that
consistent adaptive testing is possible uniformly over $s\in[s_1,s_2]$
and $L\in[L_1,L_2]$, if for all $\a,\b>0$, there is a test $\psi_\s$
depending only on $s_1,s_2,L_1,L_2,\a$ and $\b$ such that
\begin{equation}
\begin{cases}
\ \varlimsup\limits_{\s\to0} \a(\psi_\s,\thetazero) \leq \a, \\
\ \varlimsup\limits_{\s\to0} \sup\limits_{s,L} \b(\psi_\s,\thetaun^{s,L}) \leq \b.
\end{cases}
\end{equation}
However, adaptive testing is not always possible without loss of
efficiency, \ie taking $\rhosigma(s)=\rhosigma^*(s)$ for each $s$. That
is why it was suggested in \citet*{Spokoiny1996} to replace $\s$ by $\s
d_\s$ in the expression of $\rhosigma^*(s)$, where $d_\s$ is a sequence
of positive real numbers, which can be seen as a necessary payment
regarding the intensity of the noise to achieve adaptivity.

Now, we say that $\rho_{\s d_\s}^*(s), s\in[s_1,s_2]$ is the adaptive
asymptotic minimax separation rate if there are two positive constants
$C_*$ and $C^*$ such that adaptive consistent testing is impossible for
$\rhosigma(s)=\rho_{\s d_\s}^*(s)$ and $C<C_*$, and possible for
$\rhosigma(s)=\rho_{\s d_\s}^*(s)$ and $C>C^*$.

\citet*{Spokoiny1996} proved that the optimal asymptotic factor is
$(\log\log\s^{-1})^{1/4}$, for signal detection in Besov balls.
\citet*{GayraudPouet2005} extended this result for Hölder classes in
the regression model.

\citet*{FanZhangZhang2001} provided a generic tool to construct minimax
and adaptive minimax tests: the generalized maximum likelihood, that we
also use in the present work to build our procedures both in the
nonadaptive and adaptive contexts.

\subsection*{Our contribution}

The problem considered in the present work is qualitatively different
from the aforementioned works on the minimax separation rate, since our
null hypothesis is not only composite but also semiparametric.
Furthermore, it seems that the finite-dimensional parameter cannot be
uniformly consistently estimated, which contrasts with the situation of
\citet*{HorowitzSpokoiny2001}.

Nevertheless, we propose a testing procedure which is consistent when
the separation rate is of order $(\s^2\sqrt{\log\s^{-1}})^{2s/4s+1}$.
This rate is then proven to be minimax, up to a possible logarithmic
factor. Indeed, no testing procedure is consistent for a separation
rate smaller than $\s^{4s/4s+1}$, which is the rate of signal detection
in the Gaussian sequence model when the signal to be detected belongs
to a Sobolev ball and the separation from $0$ is measured by the
$l_2$-norm.

Further, an adaptive test is proposed to circumvent the limitations of
the nonadaptive approach. This test is minimax rate optimal, up to a
possible logarithmic factor, uniformly over a family of Sobolev balls.

Finally, there is a gap between our lower and upper bounds for the
asymptotic minimax separation rate. It could be argued that the lower
bound is suboptimal, and that the minimax separation rate for the
shifted curve model does contain our logarithmic factor. Indeed, the
problem of testing the goodness-of-fit of the shifted curve model can
be regarded as an adaptation to the unknown shift parameter. As a
matter of fact, if adaptation to the unknown smoothness typically
entails a loglog-factor, other types of adaptation can bring simple
logarithmic ones: it is proved in \citet*{LepskiTsybakov2000} that the
asymptotic minimax separation rate for signal detection when the signal
to be detected belongs to a Sobolev or Hölder ball and the separation
from $0$ is measured by the sup-norm is
$(\s^2\sqrt{\log\s^{-1}})^{s/2s+1}$, while it is $\s^{2s/2s+1}$ when
the separation from 0 is measured by the value of the signal at a fixed
point.  The logarithmic factor can be interpreted as a payment for the
adaptation of the problem of testing at one point when this point is
unknown. Furthermore, note that the same logarithmic factor appears in
\citet*{FromontLevyLeduc2006}, where upper bounds on the minimax
separation rate are established in the problem of periodic signal
detection with unknown period.

\subsection*{Organization of the paper}

The rest of this paper is organized as follows: a nonadaptive procedure
is proposed in Section~\ref{section:nonadaptiveupperbound}, and
adjusted in Section~\ref{section:adaptiveupperbound} to obtain an
adaptive test. We also state their minimax performances, which
Section~\ref{section:lowerbound} indicates to be at least nearly
optimal in the minimax sense. The theorems are proved in
Sections~\ref{section:proofupperbound}
to~\ref{section:prooflowerbound}, and the lemmas used in their proofs
are presented in Section~\ref{section:lemmas}. The model is discussed
in Section~\ref{section:discussion}.

\section{Nonadaptive testing procedure}\label{section:nonadaptiveupperbound}

Here, we build a test which will be proven later to be minimax, up to a
possible logarithmic factor. Indeed, the procedure achieves the rate
$(\s^2\sqrt{\log\s^{-1}})^{2s/4s+1}$.

Our proposal, which carries on the work presented in
\citet*{CollierDalalyan2012}, is based on standardized versions
$\l_{\s}(N)$ of estimators of $d(\bc,\bcdiese)$:
\begin{equation}\label{definitionofthetest}
\begin{cases}
\ \l_{\s}(N) = \quad \frac{1}{4\s^2\sqrt{N}} \min_\tau
\Big[\sum_{j=1}^{N} \big| Y_j - e^{-\ii j\tau} \Ydiese_j \big|^2 \Big] - \sqrt{N}, \phantom{\Big()} \\
\ \psi_\s(N,q) = \quad \fcar_{\{\l_\s(N) > q \}}, \phantom{\Big()}
\end{cases}
\end{equation}
for $N\in\NN^*$ and $q\in\RR$. Put into words, the test $\psi_\s(N,q)$
rejects the null hypothesis when the statistic $\l_\s(N)$ exceeds the
threshold $q$ and accepts it otherwise. The following theorem
establishes the minimax properties of this testing procedure for a
proper choice of the tuning parameters.

\begin{theorem}\label{theorem:upperbound}
Set
\begin{equation}
\begin{cases}
\ \thetazero = \Big\{ (\bc,\bcdiese)\in l_2\times l_2 \,|\, d(\bc,\bcdiese) = 0 \Big\},
\\[6pt]
\ \thetaun = \Big\{ (\bc,\bcdiese)\in \FsL\times\FsL \,|\, d(\bc,\bcdiese) \geq C\rhosigma \Big\},
\end{cases}
\end{equation}
with $s$ and $L$ are positive real numbers,
$\rhosigma=\big(\s^2\sqrt{\log\s^{-1}}\big)^\frac{2s}{4s+1}$ and $C^2>
4L^2c_{s,L}^{-2s} + \sqrt{\frac{256\,c_{s,L}}{4s+1}}$, $c_{s,L} =
(4sL^2\sqrt{4s+1})^{2/4s+1}$. Denote $\psi_\s$ the test $\psi_\s (N,q)$
defined in~(\ref{definitionofthetest}) with $N = N_\s(s,L) =
[c_{s,L}\,\rhosigma^{-1/s}]$ and $q = q_\a$, the quantile of
order $1-\a$ of the standard Gaussian distribution. Then
\begin{align}
&\varlimsup\limits_{\s\to0} \alpha(\psi_\s,\thetazero) \leq \a \,,\\
&\lim_{\s\to 0} \beta(\psi_\s,\thetaun) = 0 \,.
\end{align}
\end{theorem}

\begin{remark}
In the rest of this section and in the proof, we skip the dependence of
$N_\s(s,L)$ in $s$ and $L$ when no confusion is possible.
\end{remark}

The proof of this result is given in
Section~\ref{section:proofupperbound}. Let us now develop a brief
heuristic describing how one could have guessed the optimal value of
$\rhosigma$.

\subsection*{Heuristic for the performance of the nonadaptive procedure}

Our proof will show that, under $\hzero$, $\l_\s(N_\s)$ is bounded from
above in probability. Thus, we decide to reject the null hypothesis
when $\l_\s(N_\s)$ is larger than a constant to be chosen properly.

On the other hand, we inspect the behaviour of the statistic under the
alternative hypothesis and give a condition on $\rhosigma$ under which
the test statistic is orders of magnitude larger than a constant, so
that the procedure can have the desired power.

We derive the lower bound
\begin{align}
\l_\s(N_\s) \geq\; &\frac{1}{4\sqrt{N_\s}\s^2} \min_\tau\somNsigma |c_j-e^{-\ii j\tau}\cdiese_j|^2  - \Big|\somNsigma \frac{|\xi_j|^2+|\xidiese_j|^2-4}{4\sqrt{N_\s}}\Big| \\ &- \frac{1}{2\sqrt{N_\s}} \max_\tau \Big| \somNsigma \Re \big(e^{\ii j\tau}\xi_j\overline{\xidiese_j}\big)\Big| + \text{ negligible terms}.\nonumber
\end{align}
The proof will establish that the second term is bounded in
probability, while the third, that we call perturbative, is of order
$\sqrt{\log N_\s}$. The first term, up to a $4\sqrt{N_\s}\s^2$ factor,
is an approximation of the square of the pseudo-distance
$d(\bc,\bcdiese)$. Since $\bc$ and $\bcdiese$ lie in $\FsL$, the
remainder of the sum can be bounded from above, up to a constant
factor, by $N_\s^{-2s}$. In a nutshell, we get the heuristical lower
bound
\begin{equation*}
\l_\s(N_\s)\geq Cste \cdot \Big( \frac{d^2(\bc,\bcdiese)-Cste \cdot N_\s^{-2s}}{\sqrt{N_\s}\s^2} - O_P(\sqrt{\log N_\s}) \Big).
\end{equation*}
Consequently, the alternative is detected as soon as
\begin{equation*}
\rhosigma^2 \gg \max\Big( \s^2\sqrt{N_\s}, N_\s^{-2s}, \s^2\sqrt{N_\s\log N_\s} \Big) \sim \Big(\s^2\sqrt{\log\s^{-1}}\Big)^\frac{4s}{4s+1}.
\end{equation*}

\subsection*{Heuristic for the constant $C$}

We may now ask how small the constant $C$ can be without making our
testing procedure inefficient. This constant is only optimized for our
test, and we do not claim it to be optimal in the minimax sense.

The previous optimization shows that the test achieves its best rate
when $N_\s$ is of the order of ${\rhosigma^*}^{-1/s}$. Now, denoting
$N_\s=[c{\rhosigma^*}^{-1/s}]$, a similar heuristic can give an
optimized constant $C$ in the definition of $\thetaun$. Indeed,
Lemma~\ref{lemma:6} gives the more precise lower bound
$(C^2-4L^2c^{-2s}){\rhosigma^*}^2$ for the sum in the first term, and
we will prove the exact order of magnitude of the third to be
$\sqrt{\frac{256\,c}{4s+1}\log N_\s}$. Thus
\begin{equation*}
\l_\s(N_\s)\geq \Big( C^2-4L^2c^{-2s}-\sqrt{\frac{256\,c}{4s+1}} \Big)\sqrt{\log N_\s}.
\end{equation*}
and this leads to a minimization problem determining the choice
of $c$ (\cf Theorem~\ref{theorem:upperbound}).

\section{Adaptive testing procedure}\label{section:adaptiveupperbound}

The procedure given in the previous section possesses asymptotic
minimax optimality properties thanks to an appropriate choice of the
tuning parameter $N_\s$, but the practician needs to determine values
of $s$ and $L$ to implement the test. As it seems arbitrary and
nonintuitive to make assumptions on the smoothness of the signals, it
is necessary to design testing procedures independent of $s$ and $L$
that are nearly as good, in the minimax sense, as the procedure
proposed in the previous section.

In this section, we only assume that an interval $[s_1,s_2]$ is
available such that $c,\cdiese\in\FsL$ for some $s\in[s_1,s_2]$ and
$L\in[0,\pinf[$. We propose a testing procedure depending on $s_1$ and
$s_2$ but independent of $s$ and $L$, that achieves the same rate of
separation, \ie $\big(\s^2\sqrt{\log\s^{-1}}\big)^{2s/4s+1}$, as the
test based on the precise knowledge of $s$ and $L$. Furthermore, this
rate is achieved uniformly over the Sobolev classes $\FsL$ with
$s\in[s_1,s_2]$ and $L$ belonging to any compact interval included in
$\RR^+$.

Here is the idea of its construction. The nonadaptive testing procedure
proposed above depends on $s$ only via the tuning parameter
$N_\s(s,L)$. In the followings, we will change the definition of
$N_\s(s,L)$ to avoid the dependence on $L$ and we will write $N_\s(s)$.
Using a Bonferroni procedure like in \citet*{GayraudPouet2005} or
\citet*{HorowitzSpokoiny2001}, we consider the maximum of these tests
for several values of $N_\s(s)$, more precisely, we consider tests of
the form $\tilde\psi_\s(q) = \max_{N\in\calN} \psi_\s(N,q)$. For this
kind of test, the next proposition gives bounds for the first and
second type errors:

\begin{proposition}
Let $\calN$ be a set of positive integers and denote $\tilde\psi_\s(q)$
the test $\max_{N\in\calN} \psi_\s(N,q)$, where $\psi_\s$ is defined
in~\ref{definitionofthetest}, then
\begin{align*}
\begin{cases}
\ &\a(\tilde\psi_\s(q),\thetazero) \leq \sum_{N\in\calN} \a(\psi_\s(N,q),\thetazero)
\\[2pt]
\ & \b(\tilde\psi_\s(q),\thetaun^{s,L}) \leq \min_{N\in\calN} \b(\psi_\s(N,q),\thetaun^{s,L})  .
\end{cases}
\end{align*}
\end{proposition}

Consequently, the set $\calN$ has to be as small as possible (to
control the first kind error), but rich enough to approximate the set
of all $N_\s(s)$ for $s\in[s_1,s_2]$. We will show in the proof that
each $N\in\calN$ brings adaptation over all Sobolev balls of regularity
$s$ such that there is a $S$ such that $N=N_\s(S)$ and $S\leq s \leq
S+1/\log\s^{-1}$. Hence, we introduce the following notation leading to
a proper choice of $\calN$.

\enlargethispage{12pt}
For every $s_2 > s_1 >0$, define
\begin{equation}\label{definitionoftheadaptivetest}
\begin{cases}
\ \S(s_1,s_2) = \Big\{ s_1 + \frac{j}{\log\s^{-1}} \,|\, j\geq0, s_1 + \frac{j}{\log\s^{-1}} \leq s_2   \Big\},
\phantom{\Big()}\\[4pt]
\ \calN(s_1,s_2) = \Big\{ N_\s(s)=\big[\rhosigma^*(s)^{-1/s}\big] \,|\, s\in\S(s_1,s_2) \Big\}. \phantom{\Big()} \\
\end{cases}
\end{equation}

\begin{theorem}\label{theorem:adaptiveupperbound}
Set
\begin{equation}
\begin{cases}
\ \thetazero = \Big\{ (\bc,\bcdiese)\in l_2\times l_2 \,|\, d(\bc,\bcdiese) = 0 \Big\},
\\[4pt]
\ \thetaun^{s,L} = \Big\{ (\bc,\bcdiese)\in\FsL\times\FsL \,|\, d(\bc,\bcdiese) \geq \rhosigma(s) \Big\},
\end{cases}
\end{equation}
with $C>0$, $\rhosigma(s)=C\rhosigma^*(s)$,
$\rhosigma^*(s)=\big(\s^2\sqrt{\log\s^{-1}}\big)^\frac{2s}{4s+1}$.
\eject

Consider the test $\tilde{\psi}_\s = \max_{N\in\calN(\s_1,\s_2)}
\,\psi_\s\big(N,\sqrt{2\log\log\s^{-1}}\big)$,  where $\psi_\s$ is
defined in (\ref{definitionofthetest}). Then, for the interval
$[s_1,s_2]$ used in the construction of the test $\tilde\psi_\sigma$
and for any interval $[L_1,L_2]$ included in $\mathbb R^+_*$, there is
a constant $C$ such that
\begin{align}
&\lim_{\s\to 0} \a\big(\tilde{\psi}_\s,\thetazero\big) = 0,\\
&   \lim_{\s\to 0} \sup_{[L_1,L_2]}\sup_{[s_1,s_2]}\b\big(\tilde{\psi}_\s,\thetaun^{s,L}\big) = 0,
\end{align}
\end{theorem}

\begin{remark}
In the statement of this theorem, one observes that the constants $L_1$
and $L_2$ are not used in the definition of the test, while $L$ was, in
the definition of the nonadaptive procedure. Indeed, we optimized the
separation constant $C$ and gave an expression depending on $L$, while
this optimization was not our matter in the second theorem.
\end{remark}

\begin{remark}
The theorem claims that there exists a value of $C$ for which the first
and second type errors can be controlled. From the proof of the
theorem, we see that it is sufficient that such a constant satisfies
\begin{equation*}
\begin{cases}
\ C^2-4L_2^2e^\frac{8}{(4s_1+1)^2}-\frac{C}{2} > 0 \\
\ C > \frac{64}{\sqrt{4s_1+1}}
\end{cases},
\end{equation*}
which is verified when
$C>\max\big(\frac{64}{\sqrt{4s_1+1}},\frac{1}{4}+(\frac{1}{16}+4L_2^2e^\frac{8}{(4s_1+1)^2})^{1/2}\big)$.
\end{remark}

\subsection*{Heuristic for the performance of the adaptive procedure}

Here we explain why our adaptive procedure achieves the same rate as
the nonadaptive one. The heuristic of the previous section roughly
holds, with this difference that $\max_N \l_\s(N)$ is of loglog-order
under the null hypothesis. But this term is negligible in view of the
perturbative term, so that the performances of the test do not
deteriorate in the adaptive problem.

\section{Lower bound for the minimax rate}\label{section:lowerbound}

After stating the performance of our tests, we prove in this section
that they are at least nearly rate optimal. Indeed, we are able to
establish a lower bound for our model, by proving that the detection of
a signal lying in a Sobolev ball when the separation from $0$ is
measured by the $l_2$-norm (\cf (\ref{classicalmodel}) for a
precise definition) is simpler than ours, in the sense that every
lower bound result for this model is adaptable for our purpose.

Let us first introduce the classical signal detection problem, for
which the minimax separation rate, and even the exact separation
constants, are known:
\begin{equation}\label{classicalmodel}
\begin{cases}
\ Y_j = c_j + \s\xi_j,\quad j=1,2,\ldots, \phantom{\bigg(\bigg)} \\[-6pt]
\ \thetazero^{\bf class} = \{0\}, \phantom{\bigg(\bigg)}\\[-6pt]
\ \thetaun^{\bf class} = \Big\{ \bc\in\FsL \,\big|\, \|\bc\|_2 \geq C\rhosigma \Big\}. \phantom{\bigg(\bigg)}
\end{cases}
\end{equation}
For this model, we define the errors of first and second kind of a test $\psi^{\bf class}$ by
\begin{equation}
\begin{cases}
\ \a^{\bf class}(\psi^{\bf class},\thetazero^{\bf class}) = \sup_{\thetazero^{\bf class}} \esp_{\bc}\big(\psi^{\bf class}\big), \\[3pt]
\ \b^{\bf class}(\psi^{\bf class},\thetaun^{\bf class}) = \sup_{\thetaun^{\bf class}} \esp_{\bc}\big(1-\psi^{\bf class}\big),
\end{cases}
\end{equation}
where we denote $\prob_{\bc}$ the probability engendered by
$\bY=(Y_1,Y_2,\ldots)$ when $(c_1,c_2,\ldots) = \bc$.

\begin{theorem}\label{theorem:lowerbound}
Given the two models exposed in~(\ref{model}) and~(\ref{classicalmodel}), we have
\begin{equation}
\inf_{\psi_\a} \b(\psi_\a,\thetaun) \geq \inf_{\psi_\a^{\bf class}} \b^{\bf class}(\psi_\a^{\bf class},\thetaun^{\bf class})\,,
\end{equation}
where the infima are taken over all tests of level $\a$ respectively for our model and for the classical one.
\end{theorem}
Thus, our model can benefit from every lower bound result on
model~(\ref{classicalmodel}). We choose to exploit the nonasymptotic
results presented in \citet*{Baraud2002}, Proposition 3. The following
theorem shows that the asymptotic minimax separation rate for our
problem is not smaller than $\s^{4s/4s+1}$.
\begin{corollary*}
Let $\a$ and $\b$ be in $]0,1]$. Define $\eta= 2(1-\a-\b)$,
$\mathcal{L}=\log(1+\eta^2)$ and $\rho^2 = \sup_{d\geq1} \big[
\sqrt{2\mathcal{L}d}\s^2 \wedge L^2 d^{-2s} \big]$. Then
\begin{equation}
\rhosigma\leq\rho \quad \Rightarrow \quad \inf_{\psi^\a} \b\Big(\psi^\a,\thetaun\Big) \geq \b,
\end{equation}
where the infimum is taken over all tests of level $\a$ for the shifted curve model.
\end{corollary*}

\begin{remark}
We can approximate $\rho$ by computing
$$\sup_{x\in\RR^+} \big[
\sqrt{2\mathcal{L}x}\s^2 \wedge L^2 x^{-2s} \big] = L^\frac{1}{4s+1}
\big(\s^2\sqrt{2\mathcal{L}}\big)^\frac{2s}{4s+1}.$$
\end{remark}

\begin{remark}
Our proof shows that every lower bound result for adaptive testing
could be used for our purpose as well, for instance
\citet*{GayraudPouet2005}.
\end{remark}

\vspace*{-3pt}
\section{Discussion}\label{section:discussion}

\vspace*{-3pt}
\subsection*{Model}

\vspace*{-3pt}
The choice of our model was inspired by practical considerations, and
we intend to apply it to a problem in computer vision: that of keypoint
matching as briefly discussed in \citet*{CollierDalalyan2012}.
Accordingly, it is necessary to justify the realism of model
(\ref{model}).

\vspace*{-3pt}
\subsubsection*{Variance}

\vspace*{-3pt}
Although the theoretical analysis of this paper is carried out for the
Gaussian sequence model, the procedure we propose admits a simple
counterpart in the regression model, at least in the case of
deterministic equidistant design. According to the theory on the
asymptotic equivalence, our results hold true for this model as well,
provided that $s>1/2$ (\cf \citet*{Rohde2004}). However, in the model
of regression, it is not realistic to assume that the variance of noise
is known in advance.

Nevertheless, one can compute a consistent estimator of the variance
(\cf \citet*{Rice1984}) and plug this estimator in the testing
procedure. In an analogous setup, it is proved in
\citet*{GayraudPouet2001} for example, that this plug-in strategy
preserves the rate-optimality of the testing procedure. We believe that
a similar result can be deduced in our set-up as well.

\subsubsection*{Symmetry of the model}

In our modelization, the two parts corresponding in the Gaussian white
noise model to two different functions are treated symmetrically: the
same model, with the same variance and the same noise, applies to both.
But, in applications, the signals that we want to match with each other
are thought to have the same nature. In addition, it seems that it is
not meaningful to consider the case when the regularities of the
Sobolev balls are different for the signals: under $\hzero$, the
regularity has to be the same.

Besides, one could want to normalize both equations to get the same
variance for both sides. But, this would also change the functions,
which would not only differ from each other by a shift, but also by a
dilatation. Therefore, the application of our methodology to this case
is not straightforward. However, a detailed inspection reveals that our
results carry over to the case when we replace $\s$ by
$\max(\s,\sdiese)$.

\subsection*{Weighted estimator}

In \citet*{CollierDalalyan2012}, another estimator of
$d^2(\bc,\bcdiese)$ is used, stemming from a penalization of the
log-likelihood ratio. This could be adapted in our context by
considering the test statistic
\begin{equation}
\lambda^{\bf w}_{\s} = \quad \frac{1}{4\s^2\sqrt{N_\s}} \min_\tau \Bigg[\sum_{j=1}^{\pinf} w_j \big| Y_j - e^{-\ii j\tau} \Ydiese_j \big|^2 \Bigg] - \|{\bf w}\|_2,
\end{equation}
where ${\bf w}=(w_1,w_2,\ldots)$ is a sequence of real numbers in
$[0,1]$ depending on $\s$. Under some conditions on ${\bf w}$, our
study would undergo only few modifications, and only the optimal
constants would be changed. For simplicity sake, we chose not to
consider the weighted estimator.


\subsection*{From classical signal detection to shift testing}

A first guess to try solving our problem could be to use an estimator
$\hat\tau$ of the shift and to apply the classical signal detection
methods to the sequence $(Y_j-e^{\ii j \hat\tau}\Ydiese_j)$. But this
approach fails, since it is not possible to get any consistent
estimator of the shift. Indeed, for example, the shift may not be
identifiable. Consequently, the study of the perturbative term (\cf
first heuristic after Theorem~\ref{theorem:upperbound}) is unavoidable,
in order to take into account every possible shift. We think that this
uncertainty entails a price, \ie a supplementary factor in the minimax
separation rate.

\subsection*{Future research}

Our model is only a simple version of the curve registration problem.
In further work, we could study what happens when the signals are
shifted and dilated by considering the pseudo-distance
\begin{equation*}
\tilde{d}^2(\bc,\bcdiese) = \inf_{\tau,a} \sum_{j=1}^\pinf |c_j-a\,e^{\ii j \tau}\cdiese_j|^2.
\end{equation*}
Once again, the problem is whether it is possible to estimate the dilatation parameter consistently.

\section{Proof of Theorem~\ref{theorem:upperbound}}\label{section:proofupperbound}

\subsection{First kind error}

Here, we prove that the asymptotic first kind error of the test
$\psi_\s$ does not exceed the prescribed level $\a$. To this end,
denote $\tau^*$ a real number such that, under $\hzero$, $\forall
j\geq1,\, \cdiese_j = e^{\ii j\tau^*} c_j$. We skip the dependence of
$\tau^*$ on $\bc$ and $\bcdiese$. Using the inequality
$$\min_\tau\somNsigma \big| Y_j - e^{-\ii j\tau} \Ydiese_j \big|^2 \leq
\somNsigma \big| Y_j - e^{-\ii j\tau^*} \Ydiese_j \big|^2 = \s^2
\somNsigma \big|\xi_j - e^{-\ii j\tau^*} \xidiese_j \big|^2,$$ we get
\begin{align*}
\a(\psi_\s,\thetazero) &= \sup_{\thetazero} \pccdiese\Big( \frac{1}{4\s^2\sqrt{N_\s}} \min_\tau\somNsigma \big| Y_j - e^{-\ii j\tau} \Ydiese_j \big|^2 - \sqrt{N_\s} > q_\a \Big) \\
                &\leq \prob\Big( \frac{1}{4\sqrt{N_\s}}\somNsigma \big(\eta_j^2 + \tilde{\eta}_j^2-4\big) > q_\a \Big), \\ &\text{ where } \eta_j= \Re(\xi_j - e^{-\ii j\tau^*} \xidiese_j), \tilde{\eta}_j= \Im(\xi_j - e^{-\ii j\tau^*} \xidiese_j) \simiid \calN(0,2).
\end{align*}
Finally, using Berry-Esseen's inequality (\cf Theorem~\ref{theorem:BE}), we get
\begin{equation*}
\a(\psi_\s,\thetazero) \leq \a + \frac{1}{\sqrt{2\pi N_\s}}\,,
\end{equation*}
and this gives the desired asymptotic level.

\subsection{Second kind error}

It remains to study the second kind error of the test, and to show that
it tends to $0$. Our proof is based on the heuristic given earlier in
Section~\ref{section:nonadaptiveupperbound}: we decompose $\l_\s(N_\s)$
into several terms, and make use of their respective orders of
magnitude. The decomposition gives
\begin{align}\label{decomposition}
& 4\s^2\sqrt{N_\s}\l_\s(N_\s) \nonumber \\ &\geq \min_\tau \Big\{ \somNsigma |c_j-e^{-\ii j\tau}\cdiese_j|^2 + 2 \s\somNsigma \Re\big((c_j-e^{-\ii j\tau}\cdiese_j)(\overline{\xi_j-e^{-\ii j\tau}\xidiese_j})\big)\Big\} \\ &- \s^2\sqrt{N_\s} \Big|\somNsigma \frac{|\xi_j|^2+|\xidiese_j|^2-4}{\sqrt{N_\s}}\Big| - 2\s^2 \max_\tau \Big| \somNsigma \Re \big(e^{\ii j\tau}\xi_j\overline{\xidiese_j}\big)\Big|. \nonumber
\end{align}
For simplicity sake, we introduce some notation:
\begin{equation*}
\begin{cases}
\ D_\s(\bc,\bcdiese) = \min_\tau \Big\{ \somNsigma |c_j-e^{-\ii j\tau}\cdiese_j|^2 \\[-3pt]
{} +2 \s\somNsigma \Re\big((c_j-e^{-\ii j\tau}\cdiese_j)(\overline{\xi_j-e^{-\ii j\tau}\xidiese_j})\big)\Big\},
\phantom{\Bigg()}\\[-6pt]
\ A_\s = \Big|\somNsigma \frac{|\xi_j|^2+|\xidiese_j|^2-4}{\sqrt{N_\s}}\Big|, \phantom{\Bigg()}\\[-6pt]
\ B_\s = \max_\tau \Big| \somNsigma \Re \big(e^{\ii j\tau}\xi_j\overline{\xidiese_j}\big)\Big|,\phantom{\Bigg()}
\end{cases}
\end{equation*}
which, combined with (\ref{decomposition}), leads to:
\begin{equation*}
\b(\psi_\s,\thetaun) \leq \sup_{\thetaun} \pccdiese\bigg( D_\s(\bc,\bcdiese) - \s^2\sqrt{N_\s} A_\s -2\s^2 B_\s \leq 4 q_\a \s^2 \sqrt{N_\s} \bigg).
\end{equation*}
In addition to $c_{s,L}$, introduced in the definition of $N_\s$, we will need the constant $c'$ and $\e$, defined as
\begin{equation*}
\begin{cases}
\ c' = \sqrt{\frac{256\,c_{s,L}}{4s+1}}, \\
\ \e = \frac{1}{2}\, \big(C^2-4L^2c_{s,L}^{-2s}-\sqrt{\frac{256\,c_{s,L}}{4s+1}}\big).
\end{cases}
\end{equation*}
Separating the different terms to study them independently, we write
\begin{align*}
\b(\psi_\s,\thetaun) &\leq \sup_{\thetaun} \pccdiese\bigg( D_\s(\bc,\bcdiese) \leq (c'+\e+\frac{4q_\a \sqrt{c_{s,L}}}{\sqrt{\log\s^{-1}}})\rhosigma^2\bigg) \\&+ \prob\bigg( \s^2\sqrt{N_\s} A_\s > \e\rhosigma^2\bigg) + \prob\bigg( 2\s^2 B_\s > c'\rhosigma^2\bigg).
\end{align*}
\begin{itemize}[align=left, leftmargin=*, noitemsep]
\item Let us first study $\sup_{\thetaun} \pccdiese\big( D_\s(\bc,\bcdiese)
\leq (c'+\e+\frac{4q_\a \sqrt{c_{s,L}}}{\sqrt{\log\s^{-1}}})\rhosigma^2\big)$,
which contains the dominant term when $\rhosigma$ is too large.

Denoting $\d=\sqrt{C^2-4L^2c_{s,L}^{-2s}}$, Lemma~\ref{lemma:1} allows
to apply Lemma~\ref{lemma:2} with $x_0=\d\rhosigma$ and
$M=(c'+\e+\frac{4q_\a \sqrt{c_{s,L}}}{\sqrt{\log\s^{-1}}})\rhosigma^2$.
The choice of the parameters yields for $\s$ small enough
\begin{equation*}
\big(\frac{\d}{4} - \frac{c'+\e}{4\d} - \frac{q_\a \sqrt{c_{s,L}}}{\d\sqrt{\log\s^{-1}}}\big) \rhosigma > 0\,,
\end{equation*}
so that the second part of Lemma~\ref{lemma:2} holds:
\begin{align*}
&\sup_{\thetaun} \pccdiese\bigg( D_\s(\bc,\bcdiese) \leq (c'+\e+\frac{4q_\a \sqrt{c_{s,L}}}{\sqrt{\log\s^{-1}}})\rhosigma^2\bigg)\\ &\leq 2 \bigg(1+\d^{-1}L\,\rhosigma^{-1}\max\{1,N_\s^{1-s}\}\bigg) \\ & \times \bigg[\exp\Big\{-\Big(\d^2 - c'-\e - \frac{4 q_\a \sqrt{c_{s,L}}}{\sqrt{\log\s^{-1}}}\Big)^2 \frac{\rhosigma^2}{32\d^2\s^2}\Big\}+\exp\Big\{-\frac{\rhosigma^2\d^2}{8\s^2}\Big\}\bigg] \\ &\overset{\s\to0}{\longrightarrow} 0, \text{ since } \rhosigma/\s\to0 \text{ as } \s\to0. \phantom{\bigg()}
\end{align*}
\item Let us now turn to $\prob\big( \s^2\sqrt{N_\s} A_\s > \e\rhosigma^2\big)$.
Prior to using Berry-Esseen's inequality (\cf Theorem~\ref{theorem:BE}), we derive $\frac{\e\rhosigma^2}{4\s^2\sqrt{N_\s}} \geq \frac{\e}{4\sqrt{c_{s,L}}}\sqrt{\log\s^{-1}}$, so that, putting $x=\frac{\e}{4\sqrt{c_{s,L}}}\sqrt{\log\s^{-1}}$ into the formula of the theorem and using the bound $1-\Phi(x)\leq\frac{e^{-\frac{x^2}{2}}}{x\sqrt{2\pi}}$ for every positive $x$,
\begin{equation*}
\prob\bigg( \s^2\sqrt{N_\s} A_\s > \e\rhosigma^2\bigg) \leq \sqrt{\frac{2}{\pi N_\s}} + \sqrt{\frac{32c_{s,L}}{\pi\e^2}}\frac{\s^\frac{\e^2}{32c}}{\sqrt{\log\s^{-1}}}\to0.
\end{equation*}
\item Finally, it remains to control $\prob\big( 2\s^2 B_\s > c'\rhosigma^2\big)$.
We apply Lemma~\ref{lemma:3}:
\begin{align*}
\prob\bigg( 2\s^2 B_\s > c'\rhosigma^2\bigg) &\leq 2c (\log\s^{-1})^\frac{-1}{4s+1} \s^{\frac{c'^2}{64c}-\frac{4}{4s+1}} + e^{-N_\s/2} \\
&\leq 2c (\log\s^{-1})^\frac{-1}{4s+1} + e^{-N_\s/2} \to 0.
\end{align*}
\end{itemize}

\allowdisplaybreaks

\section{Proof of Theorem~\ref{theorem:adaptiveupperbound}}\label{section:proofadaptiveupperbound}

\subsection{Proposition 1}

Let $\calN$ be a set of positive integers and denote $\tilde\psi_\s(q)
= \max_{N\in\calN} \psi_\s(N,q)$, where $\psi_\s$ is defined
in~\ref{definitionofthetest}.

\begin{itemize}
\item Concerning the first kind error:
\begin{align*}
\a(\tilde\psi_\s(q),\thetazero) &= \sup_{(\bc,\bcdiese)\in\thetazero} \pccdiese\Big( \max_{N\in\calN} \min_\tau \sum_{j=1}^N |Y_j-e^{-\ii j\tau}\Ydiese_j|^2 > q \Big) \\
&\leq \sum_{N\in\calN} \sup_{(\bc,\bcdiese)\in\thetazero} \pccdiese\Big( \min_\tau \sum_{j=1}^N |Y_j-e^{-\ii j\tau}\Ydiese_j|^2 > q \Big) \\
&= \sum_{N\in\calN} \a(\psi_\s(N,q),\thetazero).
\end{align*}
\item Concerning the second kind error:
\begin{align*}
\b(\tilde\psi_\s(q),\thetaun^{s,L}) &= \sup_{\thetaun^{s,L}} \pccdiese\Big( \max_{N\in\calN} \min_\tau \sum_{j=1}^N |Y_j-e^{-\ii j\tau}\Ydiese_j|^2 \leq q \Big) \\
&\leq \sup_{\thetaun^{s,L}} \pccdiese\Big( \min_\tau \sum_{j=1}^N |Y_j-e^{-\ii j\tau}\Ydiese_j|^2 \leq q \Big), \forall N\in\calN ,\\
&\leq \min_{N\in\calN} \b(\psi_\s(N,q),\thetaun^{s,L}).
\end{align*}
\end{itemize}

\subsection{First kind error}

Here, we prove that the first kind error of the test $\tilde{\psi}_\s$
converges to $0$. To this end, denote $\tau^*$ a real number such that,
under $\hzero$, $\forall j\geq1,\, \cdiese_j = e^{\ii j\tau^*} c_j$. We
skip the dependence of $\tau^*$ on $\bc$ and $\bcdiese$. Using the
inequality
$$\min_\tau\somNsigma \big| Y_j - e^{-\ii j\tau} \Ydiese_j
\big|^2 \leq \somNsigma \big| Y_j - e^{-\ii j\tau^*} \Ydiese_j \big|^2
= \s^2 \somNsigma \big|\xi_j - e^{-\ii j\tau^*} \xidiese_j \big|^2,$$
we get
\begin{align*}
&\a\big(\tilde{\psi}_\s,\thetazero\big) \leq \sum_{N\in\calN(s_1,s_2)} \prob\bigg( \frac{1}{4\sqrt{N}} \sum_{j=1}^N (\eta_j^2 + \tilde{\eta}_j^2 -4) > \sqrt{\,2\,\log\log\s^{-1}} \bigg), \\
&\text{ where } \eta_j= \Re(\xi_j - e^{-\ii j\tau^*} \xidiese_j), \tilde{\eta}_j= \Im(\xi_j - e^{-\ii j\tau^*} \xidiese_j) \simiid \calN(0,2).
\end{align*}
Thus, using Berry-Esseen's inequality (\cf Theorem~\ref{theorem:BE}
with $x=\sqrt{2\ \log\log\s^{-1}}$) and the bound
$1-\Phi(x)\leq\frac{e^{-\frac{x^2}{2}}}{x\sqrt{2\pi}}$ for every
positive $x$,
\begin{align*}
\a\big(\tilde{\psi}_\s,\thetazero\big) &\leq \sum_{N\in\calN(s_1,s_2)} \Big\{\frac{1}{\sqrt{2\pi N}} + \frac{\exp(-\log\log\s^{-1})}{\sqrt{4\pi\,\log\log\s^{-1}}}\Big\} \\
                        &\leq \frac{1}{\sqrt{2\pi}} \frac{\card\,\calN(s_1,s_2)}{\sqrt{N_\s(s_2)}} + \frac{1}{\sqrt{4\pi}}\frac{\card\,\calN(s_1,s_2)}{\log\s^{-1}\sqrt{\log\log\s^{-1}}}.
\end{align*}
As $\card\,\calN(s_1,s_2)=1+\big[\,(s_2-s_1)\log\s^{-1}\big]$ is of
logarithmic order, this implies that
${\a\big(\tilde{\psi}_\s,\thetazero\big)\to0}$.
\vfill

\subsection{Second kind error}

Finally, we study the second kind error and prove that it converges to $0$.

For $s\in[s_1,s_2]$, define $S= \max\big\{ t\in\S(s_1,s_2)
\;\suchthat\; t \leq s \big\}$, where we omit the dependence of $S$ in
$s$ for simplicity sake. Note that $0\leq s-S \leq
\frac{1}{\log\s^{-1}}$. $S$ is an approximation of $s$ which will be
sufficient for our purpose according to Lemma~\ref{lemma:6}.

We introduce the notation
\begin{equation*}
\begin{cases}
\ D^s_\s(\bc,\bcdiese) = \min_\tau \Big\{ \sum_{j=1}^{N_\s(s)} |c_j-e^{-\ii j\tau}\cdiese_j|^2 \\[-3pt]
{} + 2 \s\sum_{j=1}^{N_\s(s)} \Re\big((c_j-e^{-\ii j\tau}\cdiese_j)(\overline{\xi_j-e^{-\ii j\tau}\xidiese_j})
\big)\Big\}, \phantom{\Bigg()}\\[-6pt]
\ A^s_\s = \Big|\sum_{j=1}^{N_\s(s)} \frac{|\xi_j|^2+|\xidiese_j|^2-4}{\sqrt{N_\s(s)}}\Big|, \phantom{\Bigg()}\\[-6pt]
\ B^s_\s = \max_\tau \Big| \sum_{j=1}^{N_\s(s)} \Re \big(e^{\ii j\tau}\xi_j\overline{\xidiese_j}\big)\Big|.\phantom{\Bigg()}
\end{cases}
\end{equation*}
and computations similar to those of the previous section yield
\begin{align*}
&\sup_{[L_1,L_2]}\sup_{[s_1,s_2]} \b(\tilde{\psi}_\s,\thetaun^{s,L}) \\[-3pt]
&\leq \sup_{s,L} \sup_{\thetaun^{s,L}}
\pccdiese\Big( D_\s^S(\bc,\bcdiese) \leq \s^2\sqrt{\,32\,N_\s(S)\,\log\log\s^{-1}} +
\frac{C}{2}\rhosigma^2(S)\Big) \\[-3pt] &+ \sum_{s\in\S} \prob\Big( \s^2\sqrt{N_\s(s)} A_\s^s > \frac{C}{4}\rhosigma^2(s)\Big) + \sum_{s\in\S} \prob\Big( 2\s^2 B_\s^s > \frac{C}{4}\rhosigma^2(s)\Big).
\end{align*}

\begin{itemize}[align=left, leftmargin=*, noitemsep]
\item Let us study $\sup_{s,L} \sup_{\thetaun^{s,L}} \pccdiese\big(
D_\s^S(\bc,\bcdiese) \leq \s^2\sqrt{\,32\,N_\s(S)\,\log\log\s^{-1}} + \frac{C}{2}\rhosigma^2(S)\big)$.

Lemma~\ref{lemma:6} implies $$\big(N_\s(S)+1\big)^{-2s} \leq
\rhosigma^*(S)^2 \leq e^\frac{8}{(4s_1+1)^2} \rhosigma^*(s)^2,$$ so
that, denoting $\d^2=C^2-4L^2e^\frac{8}{(4s_1+1)^2}$,
Lemma~\ref{lemma:1} allows to apply Lemma~\ref{lemma:2} with
$x_0=\d\rhosigma^*(s)$ and $M=\s^2\sqrt{32\,N_\s(S)\log\log\s^{-1}} +
\frac{C}{2}\rhosigma^2(s)$. On the other hand, the choice of $\d$
entails that for $C$ large and $\s$ small enough
$\phantom{e^\frac{8}{(4s_1+1)^2}}$
\begin{equation*}
\forall\, s\in[s_1,s_2], \quad \big(\frac{\d}{4} - \frac{C}{8\d}\big) \rhosigma^*(s) - \frac{\s^2 \sqrt{2\,N_\s(S)\log\log\s^{-1}}}{\d\rhosigma^*(s)} > 0.
\end{equation*}
Hence, applying the second part of Lemma~\ref{lemma:5}, we get an
inequality where the right-hand side converges to $0$ as $\s$ tends to
$0$:
\begin{align*}
&\sup_s \sup_{\thetaun^{s,L}} \pccdiese\Big( D_\s^S(\bc,\bcdiese) \leq \s^2\sqrt{\,32\,N_\s(S)\,\log\log\s^{-1}} + \frac{C}{2}   \rhosigma^2(S)\Big) \\ &\leq 2 \bigg(1+\d^{-1}L\,\rhosigma(s_2)^{-1}\max\{1,N_\s(s_1)^{1-s_1}\}\bigg) \\ &\times\bigg[\exp\Big\{-\Big( (\d^2 - \frac{C}{2}) \rhosigma^2(s_1) - \sqrt{32\,N_\s(s_1)\log\log\s^{-1}} \Big)^2/32\d^2\rhosigma^2(s_1)\s^2\Big\} \\ &+\exp\Big\{-\frac{\rhosigma^2(s_2)\d^2}{8\s^2}\Big\}\bigg].
\end{align*}
\item Consider the second term. Berry-Esseen's theorem (\cf Theorem~\ref{theorem:BE}) implies the following inequality, where the right-hand side converges to $0$ as $\s$ tends to $0$:
\begin{align*}
&\sum_{s\in\S} \prob\Big( \s^2\sqrt{N_\s(s)} A_\s^s > \frac{C}{4}\rhosigma^2(s)\Big) \\ &\leq \card\,\calN(s_1,s_2)\cdot\bigg[ \sqrt{\frac{2}{\pi N_\s(s_2)}} + \sqrt{\frac{128}{\pi C}} \frac{\s^{\frac{C}{128}}}{\sqrt{\log\s^{-1}}} \bigg].
\end{align*}
\item Let us turn to the third term. We apply Lemma~\ref{lemma:3} and get an inequality where once again the right-hand side converges to $0$ as $\s$ tends to $0$:
\begin{align*}
&\sum_{s\in\S} \prob\Big( 2\s^2 B_\s^s > \frac{C}{4}\rhosigma^2(s)\Big) \\ &\leq \card\,\calN(s_1,s_2) \cdot\Big[ \,2 (\log\s^{-1})^\frac{-1}{4s_2+1} \s^{\frac{C^2}{1024}-\frac{4}{4s_1+1}} + e^{-N_\s/2}\Big].
\end{align*}
\end{itemize}

\section{Proof of Theorem~\ref{theorem:lowerbound}}\label{section:prooflowerbound}

Consider a randomized test $\psi$ in the shifted curve model. We will
define a corresponding test in the classical model with smaller first
and second kind errors, and it is sufficient to establish the result.

First note that there is a measurable function $f$ with respect to the
$\s$-algebra engendered by the sequences $\bY$ and $\bYdiese$ and with
values in $[0,1]$ such that $\psi = f(\bY,\bYdiese)$. Denoting
$\bepsilon$ a sequence of i.i.d random variables $\calN(0,\s^2)$
independent from $\bY$, we define $\psi^{\bf class} =
\esp_{\bepsilon}\big(f(\bY,\bepsilon)|\bY\big)$, where
$\esp_{\bepsilon}$ is the integration with respect to the probability
engendered by $\bepsilon$. $\psi^{\bf class}$ is $\s(\bY)$-measurable
and thus constitutes a test for the classical model.

This testing procedure can be interpreted as a test in the shifted
curve model when $\bcdiese=0$. Indeed, $d(\bc,\bcdiese) = \|\bc\|_2$
when $\bcdiese=0$, so that $\thetazero^{\bf class}\times 0 \subseteq
\thetazero$ and $\thetaun^{\bf class} \times 0 \subseteq \thetaun$. By
Tonelli-Fubini's theorem, $\psi^{\bf class}$ satisfies
\begin{align*}
\a^{\bf class}(\psi^{\bf class},\thetazero^{\bf class}) &= \sup_{\thetazero^{\bf class}} \esp_{\bc}\big( \psi^{\bf class} \big) \\
             &= \sup_{\thetazero^{\bf class}} \esp_{\bc,0} \big( f(\bY,\bYdiese) \big) \\
             &\leq \a(\psi,\thetazero).
\end{align*}
A similar inequality holds concerning the second kind error.

\section{Lemmas}\label{section:lemmas}

\begin{lemma}\label{lemma:1}
Let $\bc=(c_1,c_2,\ldots)$ and
$\tilde{\bc}=(\tilde{c}_1,\tilde{c}_2,\ldots)$ in $\FsL$, with $s>0$,
be such that $d(\bc,\tilde{\bc})\geq C\rho$, and let $N+1 \geq
c\rho^{-1/s}$. Then
\begin{equation*}
\min_\tau \sum_{j=1}^N |c_j-e^{-\ii j\tau}\tilde{c}_j|^2 \geq (C^2 - 4L^2c^{-2s}) \rho^2.
\end{equation*}
\end{lemma}
\begin{proof}[Proof of Lemma~\ref{lemma:1}]
Since both $\bc$ and $\tilde{\bc}$ belong to the Sobolev ball, it holds that
\begin{align*}
\sum_{j>N} |c_j-e^{-ij\tau}\tilde c_j|^2 &\le
\sum_{j>N} \big(2|c_j|^2+2|\tilde c_j|^2\big) \\
&\le  2(N+1)^{-2s}\sum_{j>N} j^{2s}\big(|c_j|^2+|\tilde c_j|^2\big)\\
&\le  4L^2(N+1)^{-2s}.
\end{align*}
Consequently, taking into account that $\sum_{j=1}^\infty
|c_j-e^{-ij\tau}\tilde c_j|^2\ge d^2(c,\tilde c)\ge C^2\rho^2$, we get
\begin{align*}
\sum_{j=1}^N |c_j-e^{-ij\tau}\tilde c_j|^2 &=\sum_{j=1}^\infty |c_j-e^{-ij\tau}\tilde c_j|^2-\sum_{j>N}|c_j-e^{-ij\tau}\tilde c_j|^2\\
&\ge C^2\rho^2-4L^2(N+1)^{-2s},
\end{align*}
and the result follows in view of $N+1\ge c\rho^{-1/s}$.
\end{proof}

\begin{lemma}\label{lemma:2}
Let $N$ be some positive integer, let $\xi_j$, $\tilde\xi_j$,
$j=1,\ldots,N$ be independent complex valued random variables such that
their real and imaginary parts are independent standard Gaussian
variables, and let $\bc=(c_1,\ldots,c_N)$,
$\tilde{\bc}=(\tilde{c}_1,\ldots,\tilde{c}_N)$ be complex vectors.
Denote $\bxi=(\xi_1,\ldots,\xi_N)$,
$\btildexi=(\tilde{\xi}_1,\ldots,\tilde{\xi}_N)$ and
\begin{align*}
\begin{cases}
\ D_{\s,N}(\bc,\tilde{\bc}) = \min_\tau \Big\{ \sum_{j=1}^N |c_j-e^{-\ii j\tau}\tilde{c}_j|^2 \\ \ + 2 \s\sum_{j=1}^N \Re\big((c_j-e^{-\ii j\tau}\tilde{c}_j)(\overline{\xi_j-e^{-\ii j\tau}\tilde\xi_j})\big)\Big\}, \\
\ d_{N,\tau}(\bc,\tilde{\bc}) = \sqrt{\sum_{j=1}^N \big|c_j-e^{-\ii j\tau}\tilde{c}_j\big|^2}, \\
\ u_N(\bxi,\bc,\tilde{\bc}) = \sup_\tau \Big|\sum_{j=1}^N \frac{\Re \big[\xi_j(\overline{c_j-e^{-\ii j\tau}\tilde{c}_j})\big]}{d_{N,\tau}(\bc,\tilde{\bc})}\Big|.
\end{cases}
\end{align*}
Assume that $x_0\leq \min_\tau d_{N,\tau}(\bc,\tilde{\bc})$, then
\begin{align*}
 \forall M\in\RR, \quad &\prob\bigg( D_{\s,N}(\bc,\tilde{\bc})\leq M \bigg) \\ &\leq 2 \,\prob\bigg( \s u_N(\bxi,\bc,\tilde{\bc}) \geq \frac{x_0}{4} - \frac{M}{4x_0} \bigg) + 2 \,\prob\bigg( \frac{x_0}{2} < \s u_N(\bxi,\bc,\tilde{\bc}) \bigg).
\end{align*}
Assume further that $\bc$ and $\tilde{\bc}$ are in $\FsL$ and that
$\frac{x_0}{4} - \frac{M}{4x_0}>0$, then combining the last result with
Lemma~\ref{lemma:5},
\begin{align*}
\prob\bigg( D_{\s,N}(\bc,\tilde{\bc})\leq M \bigg) \leq & \ 2 \Big(1+x_0^{-1}L\, \max\{1,N^{1-s}\}\Big) \\ &\times \Big(\exp\big\{-(x_0^2-M)^2/32x_0^2\s^2\big\}+\exp\big\{-x_0^2/8\s^2\big\}\Big).
\end{align*}
\end{lemma}

\begin{proof}[Proof of Lemma~\ref{lemma:2}]
\begin{align*}
&\sum_{j=1}^N \big|c_j-e^{-\ii j\tau}\tilde{c}_j\big|^2 + 2\s \sum_{j=1}^N \Re \Big((c_j-e^{-\ii j\tau}\tilde{c}_j)(\overline{\xi_j-e^{-\ii j\tau}\tilde\xi_j})\Big) \phantom{\Bigg()} \\
&= d^2_{N,\tau}(\bc,\tilde{\bc}) + 2\s d_{N,\tau}(\bc,\tilde{\bc}) \sum_{j=1}^N \frac{\Re \big[\xi_j(\overline{c_j-e^{-\ii j\tau}\tilde{c}_j})\big]}{d_{N,\tau}(\bc,\tilde{\bc})} \\ &+ 2\s d_{N,\tau}(\bc,\tilde{\bc}) \sum_{j=1}^N \frac{\Re \big[\tilde\xi_j(\overline{e^{\ii j\tau}c_j-\tilde{c}_j})\big]}{d_{N,\tau}(\bc,\tilde{\bc})} \\
&\geq d^2_{N,\tau}(\bc,\tilde{\bc}) - 2\s d_{N,\tau}(\bc,\tilde{\bc}) \sup_\tau \Big|\sum_{j=1}^N \frac{\Re \big[\xi_j(\overline{c_j-e^{-\ii j\tau}\tilde{c}_j})\big]}{d_{N,\tau}(\bc,\tilde{\bc})}\Big| \\ &- 2\s d_{N,\tau}(\bc,\tilde{\bc}) \sup_\tau \Big| \sum_{j=1}^N \frac{\Re \big[\tilde\xi_j(\overline{e^{\ii j\tau}c_j-\tilde{c}_j})\big]}{d_{N,\tau}(\bc,\tilde{\bc})} \Big|. \phantom{\Bigg()}
\end{align*}
With the notation $u_N(\bxi,\bc,\tilde{\bc}) = \sup_\tau \Big|\sum_{j=1}^N \frac{\Re \big[\xi_j(\overline{c_j-e^{-\ii j\tau}\tilde{c}_j})\big]}{d_{N,\tau}(\bc,\tilde{\bc})}\Big|$, we obtain
\begin{equation*}
D_{\s,N}(\bc,\tilde{\bc}) \geq \min_{x\geq x_0}(x^2-ax),
\end{equation*}
with $a=2\s u_N(\bxi,\bc,\tilde{\bc}) + 2\s
u_N(\tilde{\bxi},\tilde{\bc},\bc)$. Now, using the fact that
$\min_{x\geq x_0}(x^2-ax)$ is reached at the point $x_0$ if
$x_0\geq\frac{a}{2}$, we get
\begin{align*}
\prob\bigg( D_{\s,N}(\bc,\tilde{\bc})\leq M \bigg) &\leq \prob\bigg( x_0^2 - 2 x_0 \s u_N(\bxi,\bc,\tilde{\bc}) - 2 x_0 \s u_N(\tilde{\bxi},\tilde{\bc},\bc) \leq M \bigg) \\
&+ \prob\bigg( x_0 < \s u_N(\bxi,\bc,\tilde{\bc}) + \s u_N(\tilde{\bxi},\tilde{\bc},\bc)) \bigg)\\
&\leq 2 \,\prob\bigg( \s u_N(\bxi,\bc,\tilde{\bc}) \geq \frac{x_0}{4} - \frac{M}{4x_0} \bigg)\\ &+ 2 \,\prob\bigg( \frac{x_0}{2} < \s u_N(\bxi,\bc,\tilde{\bc})\bigg),
\end{align*}
since $u_N(\bxi,\bc,\tilde{\bc})$ and $u_N(\tilde{\bxi},\tilde{\bc},\bc)$ have the same distribution.
\end{proof}

\begin{lemma}\label{lemma:3}
Let $\xi_j,\tilde{\xi}_j$ be independent complex valued random
variables such that their real and imaginary parts are independent
standard Gaussian variables, let $c$, $s$ and $\s$ be some positive
real numbers. Denote
\begin{equation*}
\begin{cases}
\ \rhosigma=(\s^2\sqrt{\log\s^{-1}})^\frac{2s}{4s+1}, \\
\ N_\s=[c\rhosigma^{-1/s}], \\
\ B_\s=\max_\tau \Big| \somNsigma \Re \big(e^{\ii j\tau}\xi_j\tilde{\xi}_j\big)\Big|.
\end{cases}
\end{equation*}
Then, for $\s$ small enough and for every positive $c'$,
\begin{equation*}
\prob\bigg( 2\s^2 B_\s > c'\rhosigma^2\bigg) \leq 2c (\log\s^{-1})^\frac{-1}{4s+1} \s^{\frac{c'^2}{64c}-\frac{4}{4s+1}}+ e^{-N_\s/2}.
\end{equation*}
\end{lemma}

\begin{proof}[Proof of Lemma~\ref{lemma:3}]
Applying Lemma~\ref{lemma:4}, we state that, for $\s$ small enough,
\begin{equation*}
\prob\Big(B_\s > 4 x \sqrt{N_\s\log(\s^{-1})} \Big) \leq 2c (\log\s^{-1})^\frac{-1}{4s+1} \s^{x^2-\frac{4}{4s+1}} + e^{-N_\s/2},
\end{equation*}
from which follows that
\begin{equation*}
\prob\Big(B_\s > 4 x \rhosigma^{-1/2s}\sqrt{c\log(\s^{-1})} \Big) \leq 2c (\log\s^{-1})^\frac{-1}{4s+1} \s^{x^2-\frac{4}{4s+1}} + e^{-N_\s/2}.
\end{equation*}
We conclude, observing that $4 x \rhosigma^{-1/2s}\sqrt{c\log(\s^{-1})}
= \frac{8x\rhosigma^2\sqrt{c}}{2\s^2}$.
\end{proof}

\begin{lemma}\label{lemma:4}
Let $N$ be some positive integer and let $\xi_j$, $\tilde{\xi}_j$,
$j=1,\ldots,N$, be independent complex valued random variables such
that their real and imaginary parts are independent standard Gaussian
variables. Let $\bu=(u_1,\ldots,u_N)$ be a vector of real numbers.
Denote $S(t) = \sum_{j=1}^{N} u_j \Re\big(e^{\ii
jt}\xi_j\tilde{\xi}_j\big)$ for every $t$ in $[0,2\pi]$ and
$\|S\|_\infty= \sup_{t\in[0,2\pi]} |S(t)|$. Then
\begin{equation*}
\forall x,y>0, \quad \prob\Big(\|S\|_\infty> \sqrt{2} x\big(\|\bu\|_2+y\|\bu\|_\infty\big) \Big) \leq (N+1)e^{-x^2/2}+e^{-y^2/2}.
\end{equation*}
\end{lemma}

\begin{proof}[Proof of Lemma~\ref{lemma:4}]
We refer to \citet*{CollierDalalyan2012}, Lemma 3, for a proof of this lemma.
\end{proof}

\begin{lemma}\label{lemma:5}
Let $\bc=(c_1,c_2,\ldots)$ and
$\tilde{\bc}=(\tilde{c}_1,\tilde{c}_2,\ldots)$ in $\FsL$ with $s>0$ and
let $N$ be an integer. Denoting $\eta_j, \tilde{\eta}_j \simiid
\nzeroun$, we define $$S(t) = \sum_{j=1}^N \frac{\eta_j \Re(c_j-e^{-\ii
jt}\tilde{c}_j) + \tilde{\eta}_j \Im(c_j-e^{-\ii
jt}\tilde{c}_j)}{\sqrt{\sum_{j=1}^N \big|c_j-e^{-\ii
jt}\tilde{c}_j\big|^2}}$$ for every $t$ in $[0,2\pi]$. Then
\begin{equation*}
\prob\bigg( \|S\|_\infty \geq x\bigg) \leq \Big( \frac{L\cdot\max\{1,N^{1-s}\}}{\sqrt{\min_\tau \sum_{j=1}^N |c_j-e^{-\ii j\tau}\tilde{c}_j|^2}} +1 \Big) e^{-\frac{x^2}{2}}.
\end{equation*}
\end{lemma}

First recall Berman's formula, that we will need in the proof.

\begin{theorem}[\citet*{Berman1988}]\label{theorem:Berman}
Let $N$ be a positive integer, $a<b$ some real numbers and $g_j$,
$j=1,\ldots,N$ be continuously differentiable functions on $[a,b]$
satisfying $\sum_{j=1}^N g_j(t)^2 = 1$ for all $t\in\RR$ and $\eta_j$,
$j=1,\ldots,N$, some independent standard Gaussian variables. Then
\begin{equation*}
\prob\bigg( \sup_{[a,b]} \sum_{j=1}^N g_j(t) \eta_j \geq x\bigg) \leq \frac{I}{2\pi}e^{-\frac{x^2}{2}} + \int_x^\infty \frac{e^{-\frac{t^2}{2}}}{\sqrt{2\pi}}\,dt
\end{equation*}
with
\begin{equation*}
I = \int_a^b \bigg[{\sum_{j=1}^N g_j'(t)^2}\bigg]^{1/2} \,dt.
\end{equation*}
\end{theorem}

\begin{proof}[Proof of Lemma~\ref{lemma:5}]
Denote
\begin{equation*}
\begin{cases}
\ f_j(t) = \frac{\Re(c_j-e^{-\ii j t}\tilde{c}_j)}{\sqrt{\sum_{k=1}^{N} |c_k-e^{-\ii kt}\tilde{c}_k|^2}},
\\[6pt]
\ g_j(t) = \frac{\Im(c_j-e^{-\ii j t}\tilde{c}_j)}{\sqrt{\sum_{k=1}^{N} |c_k-e^{-\ii kt}\tilde{c}_k|^2}}.
\end{cases}
\end{equation*}
We compute the derivatives of these functions:
\begin{align*}
f_j'(t) &= \frac{-\Im( je^{-\ii j t}\tilde{c}_j)}{\sqrt{\sum_{k=1}^{N} |c_k-e^{-\ii kt}\tilde{c}_k|^2}} \\
&+ \frac{\Re(c_j-e^{-\ii j t}\tilde{c}_j)}{\big(\sum_{k=1}^{N} |c_k-e^{-\ii kt}\tilde{c}_k|^2\big)^\frac{3}{2}} \sum_{k=1}^{N} \Im(k\overline{c}_k\tilde{c}_ke^{-\ii kt}) \\
\text{and } g_j'(t) &= \frac{\Re( je^{-\ii j t}\tilde{c}_j)}{\sqrt{\sum_{k=1}^{N}
|c_k-e^{-\ii kt}\tilde{c}_k|^2}} \\
&+ \frac{\Im (c_j-e^{-\ii j t}\tilde{c}_j)}{\big(\sum_{k=1}^{N} |c_k-e^{-\ii kt}\tilde{c}_k|^2\big)^\frac{3}{2}} \sum_{k=1}^{N} \Im(k\overline{c}_k\tilde{c}_ke^{-\ii kt}),
\end{align*}
whence
\begin{align*}
\sum_{j=1}^{N} \big(f_j'(t)^2 + g_j'(t)^2\big) &= \frac{\sum_{j=1}^{N} j^2 |\tilde{c}_j|^2}{\sum_{k=1}^{N} |c_k-e^{-\ii kt}\tilde{c}_k|^2} - \bigg( \frac{\sum_{k=1}^{N} \Im(k\overline{c}_k\tilde{c}_ke^{-\ii kt})}{\sum_{k=1}^{N} |c_k-e^{-\ii kt}\tilde{c}_k|^2} \bigg)^2 \\
&\leq \frac{L^2 \max\{1,N^{2-2s}\} }{\min_{t} \sum_{k=1}^{N} |c_k-e^{-\ii kt}\tilde{c}_k|^2}
\end{align*}
The conclusion follows from Berman's formula.
\end{proof}

\begin{lemma}\label{lemma:6}
Let $\s$ be a positive real number and $s,S$ in
$[s_1,s_2]\subseteq\RR^+_*$ be such that $0\leq s-S \leq
\frac{1}{\log\s^{-1}}$. Denote $\rhosigma^*(s)=\big(\s^2
\sqrt{\log\s^{-1}}\big)^\frac{2s}{4s+1}$, then, for $\s$ small enough,
\begin{equation*}
\frac{\rhosigma^*(S)}{\rhosigma^*(s)} \leq e^\frac{4}{(4s_1+1)^2}.
\end{equation*}
\end{lemma}

\begin{proof}[Proof of Lemma~\ref{lemma:6}]
By the definition of $\rhosigma^*(s)$, we have
\begin{equation*}
\frac{\rhosigma^*(S)}{\rhosigma^*(s)} = \Big(\s^2 \sqrt{\log(\s^{-1})}\Big)^\frac{2(S-s)}{(4s+1)(4S+1)},
\end{equation*}
which, when $\s$ is so small that $\s^2\sqrt{\log\s^{-1}}\le1$, leads, with the hypothesis on $s$ and $S$,
\begin{equation*}
\frac{\rhosigma^*(S)}{\rhosigma^*(s)} \le \Big(\s^2 \sqrt{\log(\s^{-1})}\Big)^\frac{-2}{(4s_1+1)^2\log\s^{-1}}.
\end{equation*}
Then, we compute
\begin{align*}
&\Big(\s^2 \sqrt{\log(\s^{-1})}\Big)^\frac{-2}{(4s_1+1)^2\log\s^{-1}} \\
&= \exp\Big\{ \frac{-2}{(4s_1+1)^2\log\s^{-1}} ( 2\log\s + \frac{1}{2}\log\log\s^{-1} ) \Big\} \\
&= \exp\Big\{ \frac{4}{(4s_1+1)^2} (1 - \frac{\log\log\s^{-1}}{4\log\s^{-1}} ) \Big\} \\
&\le e^\frac{4}{(4s_1+1)^2},
\end{align*}
and this concludes the proof.
\end{proof}

Finally, we recall here Berry-Esseen's inequality, in a simpler version
than Theorem 5.4 of \citet*{Petrov1995}.
\begin{theorem}[Berry-Esseen's inequality]\label{theorem:BE}
Let $N$ be a positive integer and some random variables
$X_1,\ldots,X_N\simiid X$ and such that $\esp(X) = 0$, $\var(X) =
\g^2$, $\esp|X|^3 = m^3 < \pinf.$ Denote $F_N(x) =
\prob\big(\frac{1}{\sqrt{N}\g}\sum_{j=1}^N X_j <x\big)$ and $\Phi$ the
distribution function of the standard Gaussian variable. Then
\begin{equation*}
\sup_x |F_N(x) - \Phi(x)| \leq \frac{A m^3}{\g^3} \frac{1}{\sqrt{N}}\,,
\end{equation*}
for an absolute constant number $A$. Moreover, in the case when
$X=Y^2-1$ and $Y$ has a centered Gaussian distribution, and using the
majoration $A\le \frac{1}{2}$,
\begin{equation*}
\sup_x |F_N(x) - \Phi(x)| \leq \frac{1}{\sqrt{2\pi N}}\,.
\end{equation*}
\end{theorem}


\begin{thebibliography}{37}
\providecommand{\natexlab}[1]{#1}
\providecommand{\url}[1]{\texttt{#1}}
\expandafter\ifx\csname urlstyle\endcsname\relax
  \providecommand{\doi}[1]{doi: #1}\else
  \providecommand{\doi}{doi: \begingroup \urlstyle{rm}\Url}\fi

\bibitem[Baraud(2002)]{Baraud2002}
\textsc{Y.~Baraud.}
\newblock Non-asymptotic minimax rates of testing in signal detection.
\newblock \emph{Bernoulli}, 8(5):577--606, 2002.
\MR{1935648}

\bibitem[Baraud et~al.(2003)Baraud, Huet, and Laurent]{BaraudHuetLaurent2003}
\textsc{Y.~Baraud, S.~Huet, and B.~Laurent.}
\newblock Adaptive tests of linear hypotheses by model selection.
\newblock \emph{Ann. Statist.}, 31(1):225--251, 2003.
\MR{1962505}

\bibitem[Baraud et~al.(2005)Baraud, Huet, and Laurent]{BaraudHuetLaurent2005}
\textsc{Y.~Baraud, S.~Huet, and B.~Laurent.}
\newblock Testing convex hypotheses on the mean of a {G}aussian vector.
  {A}pplication to testing qualitative hypotheses on a regression function.
\newblock \emph{Ann. Statist.}, 33(1):214--257, 2005.
\MR{2157802}

\bibitem[Berman(1988)]{Berman1988}
\textsc{S.~M. Berman.}
\newblock Sojourns and extremes of a stochastic process defined as a random
  linear combination of arbitrary functions.
\newblock \emph{Comm. Statist. Stochastic Models}, 4(1):  1--43, 1988.
\MR{0938574}

\bibitem[Bigot and Gadat(2010)]{BigotGadat2010}
\textsc{J.~Bigot and S.~Gadat.}
\newblock A deconvolution approach to estimation of a common shape in a shifted
  curves model.
\newblock \emph{Ann. Statist.}, 38(4):2422--2464, 2010.
\MR{2676894}

\bibitem[Bigot et~al.(2009{\natexlab{a}})Bigot, Gadat, and  Loubes]{BigotGadatLoubes2009}
\textsc{J.~Bigot, S.~Gadat, and J.-M. Loubes.}
\newblock Statistical {M}-estimation and consistency in large deformable models
  for image warping.
\newblock \emph{J. Math. Imaging Vis.}, 34(3):270--290,
  2009{\natexlab{a}}.
\MR{2515449}

\bibitem[Bigot et~al.(2009{\natexlab{b}})Bigot, Gamboa, and  Vimond]{BigotGamboaVimond2009}
\textsc{J.~Bigot, F.~Gamboa, and M.~Vimond.}
\newblock Estimation of translation, rotation, and scaling between noisy images
  using the {F}ourier-{M}ellin transform.
\newblock \emph{SIAM J. Imaging Sci.}, 2(2):614--645,
  2009{\natexlab{b}}.
\MR{2519925}

\bibitem[Brown and Low(1996)]{BrownLow1996}
\textsc{L.~D. Brown and M.~G. Low.}
\newblock Asymptotic equivalence of nonparametric regression and white noise.
\newblock \emph{Ann. Statist.}, 24(6):2384--2398, 1996.
\MR{1425958}

\bibitem[Butucea and Tribouley(2006)]{ButuceaTribouley2006}
\textsc{C.~Butucea and K.~Tribouley.}
\newblock Nonparametric homogeneity tests.
\newblock \emph{Journal of statistical planning and inference}, 136  (3):597--639, 2006.
\MR{2181971}

\bibitem[Castillo and Loubes(2009)]{CastilloLoubes2009}
\textsc{I.~Castillo and J.-M. Loubes.}
\newblock Estimation of the distribution of random shifts deformation.
\newblock \emph{Math. Meth. Statist.}, 18(1):21--42, 2009.
\MR{2508947}

\bibitem[Collier and Dalalyan(2012)]{CollierDalalyan2012}
\textsc{O.~Collier and A.~S. Dalalyan.}
\newblock Wilks' phenomenon and penalized likelihood-ratio test for
  nonparametric curve registration.
\newblock \emph{Journal of Machine Learning Research - Proceedings Track},
  22:264--272, 2012.

\bibitem[Dalalyan and Rei{\ss}(2006)]{DalalyanReiss2006}
\textsc{A.~Dalalyan and M.~Rei{\ss}.}
\newblock Asymptotic statistical equivalence for scalar ergodic diffusions.
\newblock \emph{Probab. Theory Related Fields}, 134(2):  248--282, 2006.
\MR{2222384}

\bibitem[Dalalyan(2007)]{Dalalyan2007}
\textsc{A.~S. Dalalyan.}
\newblock Stein shrinkage and second-order efficiency for semiparametric
  estimation of the shift.
\newblock \emph{Math. Methods Statist.}, 16(1):42--62,
  2007.
\MR{2319470}

\bibitem[Dalalyan et~al.(2006)Dalalyan, Golubev, and  Tsybakov]{DalalyanGolubevTsybakov2006}
\textsc{A.~S. Dalalyan, G.~K. Golubev, and A.~B. Tsybakov.}
\newblock Penalized maximum likelihood and semiparametric second-order
  efficiency.
\newblock \emph{Ann. Statist.}, 34(1):169--201, 2006.
\MR{2275239}

\bibitem[Ermakov(1990)]{Ermakov1990}
\textsc{M.~S. Ermakov.}
\newblock Minimax detection of a signal in {G}aussian white noise.
\newblock \emph{Teor. Veroyatnost. i Primenen.}, 35(4):  704--715, 1990.
\MR{1090496}

\bibitem[Ermakov(1996)]{Ermakov1996}
\textsc{M.~S. Ermakov.}
\newblock Asymptotically minimax criteria for testing complex nonparametric
  hypotheses.
\newblock \emph{Problems Inform. Transmission}, 33:184--196, 1996.
\MR{1441739}

\bibitem[Fan et~al.(2001)Fan, Zhang, and Zhang]{FanZhangZhang2001}
\textsc{J.~Fan, C.~Zhang, and J.~Zhang.}
\newblock Generalized likelihood ratio statistics and wilks phenomenon.
\newblock \emph{Ann. Statist.}, 29(1):153--193, 2001.
\MR{1833962}

\bibitem[Fromont and L\'evy-Leduc(2006)]{FromontLevyLeduc2006}
\textsc{M.~Fromont and C.~L\'evy-Leduc.}
\newblock Adaptive tests for periodic signal detection with applications to
  laser vibrometry.
\newblock \emph{ESAIM P. S.}, 10:46--75 (electronic), 2006.
\MR{2197102}

\bibitem[Gamboa et~al.(2007)Gamboa, Loubes, and Maza]{GamboaLoubesMaza2007}
\textsc{F.~Gamboa, J.-M. Loubes, and E.~Maza.}
\newblock Semi-parametric estimation of shifts.
\newblock \emph{Electron. J. Statist.}, 1:616--640, 2007.
\MR{2369028}

\bibitem[Gayraud and Pouet(2001)]{GayraudPouet2001}
\textsc{G.~Gayraud and C.~Pouet.}
\newblock Minimax testing composite null hypotheses in the discrete regression
  scheme.
\newblock \emph{Math. Methods Statist.}, 10(4):375--394
  (2002), 2001.
\newblock Meeting on Mathematical Statistics (Marseille, 2000).
\MR{1887339}

\bibitem[Gayraud and Pouet(2005)]{GayraudPouet2005}
\textsc{G.~Gayraud and C.~Pouet.}
\newblock Adaptive minimax testing in the discrete regression scheme.
\newblock \emph{Probab. Theory Relat. Fields}, 133(4):  531--558, 2005.
\MR{2197113}

\bibitem[Grama and Nussbaum(1998)]{GramaNussbaum1998}
\textsc{I.~Grama and M.~Nussbaum.}
\newblock Asymptotic equivalence for nonparametric generalized linear models.
\newblock \emph{Probability Theory and Related Fields}, 111  (2):167--214, 1998.
\MR{1633574}

\bibitem[Grama and Nussbaum(2002)]{GramaNussbaum2002}
\textsc{I.~Grama and M.~Nussbaum.}
\newblock Asymptotic equivalence for nonparametric regression.
\newblock \emph{Math. Methods Statist.}, 11(1):1--36, 2002.
\MR{1900972}

\bibitem[Horowitz and Spokoiny(2001)]{HorowitzSpokoiny2001}
\textsc{J.~L. Horowitz and V.~G. Spokoiny.}
\newblock An adaptive, rate-optimal test of a parametric mean-regression model
  against a nonparametric alternative.
\newblock \emph{Econometrica}, 69(3):599--631, 2001.
\MR{1828537}

\bibitem[Ingster(1982)]{Ingster1982}
\textsc{Y.~I. Ingster.}
\newblock Minimax nonparametric detection of signals in white {G}aussian noise.
\newblock \emph{Problems of Information Transmission}, 18:130--140,
  1982.
\MR{0689340}

\bibitem[Ingster(1993)]{Ingster1993}
\textsc{Y.~I. Ingster.}
\newblock Asymptotically minimax hypothesis testing for nonparametric
  alternatives. {I,II,III}.
\newblock \emph{Math. Methods Statist.}, 2(2):85--114,
  1993.
\MR{1257978}

\bibitem[Ingster and Suslina(2003)]{IngsterSuslina2003}
\textsc{Y.I. Ingster and I.A. Suslina.}
\newblock \emph{{Nonparametric goodness-of-fit testing under Gaussian models}}.
\newblock Springer Verlag, 2003.
\MR{1991446}

\bibitem[Isserles et~al.(2011)Isserles, Ritov, and Trigano]{TriganoIsserlesRitov2011}
\textsc{U.~Isserles, Y.~Ritov, and T.~Trigano.}
\newblock Semiparametric curve alignment and shift density estimation for
  biological data.
\newblock \emph{IEEE Transactions on Signal Processing, in press}, 2011.

\bibitem[Lepski and Spokoiny(1999)]{LepskiSpokoiny1999}
\textsc{O.~V. Lepski and V.~G. Spokoiny.}
\newblock Minimax nonparametric hypothesis testing: the case of an
  inhomogeneous alternative.
\newblock \emph{Bernoulli}, 5(2):333--358, 1999.
\MR{1681702}

\bibitem[Lepski and Tsybakov(2000)]{LepskiTsybakov2000}
\textsc{O.~V. Lepski and A.~B. Tsybakov.}
\newblock Asymptotically exact nonparametric hypothesis testing in sup-norm and
  at a fixed point.
\newblock \emph{Probability Theory and Related Fields}, 117  (1):17--48, 2000.
\MR{1759508}

\bibitem[Lowe(2004)]{Lowe2004}
\textsc{D.G. Lowe.}
\newblock {Distinctive image features from scale-invariant keypoints}.
\newblock \emph{International journal of computer vision}, 60  (2):91--110, 2004.

\bibitem[Nussbaum(1996)]{Nussbaum1996}
\textsc{M.~Nussbaum.}
\newblock Asymptotic equivalence of density estimation and {G}aussian white
  noise.
\newblock \emph{Ann. Statist.}, 24(6):2399--2430, 1996.
\MR{1425959}

\bibitem[Petrov(1995)]{Petrov1995}
\textsc{V.~V. Petrov.}
\newblock \emph{Limit theorems of probability theory}, volume~4 of \emph{Oxford
  Studies in Probability}.
\newblock The Clarendon Press Oxford University Press, New York, 1995.
\newblock Sequences of independent random variables, Oxford Science
  Publications.
\MR{1353441}

\bibitem[Rei{\ss}(2008)]{Reiss2008}
\textsc{M.~Rei{\ss}.}
\newblock Asymptotic equivalence for nonparametric regression with multivariate
  and random design.
\newblock \emph{Ann. Statist.}, 36(4):1957--1982, 2008.
\MR{2435461}

\bibitem[Rice(1984)]{Rice1984}
\textsc{J.~Rice.}
\newblock Bandwidth choice for nonparametric regression.
\newblock \emph{The Annals of Statistics}, 12(4):  1215--1230, 1984.
\MR{0760684}

\bibitem[Rohde(2004)]{Rohde2004}
\textsc{A.~Rohde.}
\newblock On the asymptotic equivalence and rate of convergence of
  nonparametric regression and gaussian white noise.\vadjust{\eject}
\newblock \emph{Statistics \& Decisions/International mathematical journal for
  stochastic methods and models}, 22(3/2004):235--243,
  2004.
\MR{2125610}

\bibitem[Spokoiny(1996)]{Spokoiny1996}
\textsc{V.~G. Spokoiny.}
\newblock Adaptive hypothesis testing using wavelets.
\newblock \emph{Ann. Statist.}, 24(6):2477--2498, 1996.
\MR{1425962}

\end{thebibliography}
\end{document}